\documentclass[10pt, a4paper]{scrartcl}

\usepackage[T1]{fontenc}
\usepackage{lmodern}
\usepackage[english]{babel}
\usepackage{amsfonts}
\usepackage{amsmath} 
\usepackage{amssymb}
\usepackage{amsthm}
\usepackage{mathrsfs}
\usepackage[x11names,dvipsnames,svgnames]{xcolor}
\usepackage{booktabs}
\usepackage{float}
\usepackage{enumitem}
\usepackage{graphicx}
\usepackage{tikz,pgf}
\usetikzlibrary{calc}
\usetikzlibrary{decorations.markings}
\usetikzlibrary{shapes.geometric}
\usetikzlibrary{patterns}
\usetikzlibrary{intersections}
\usepackage{multirow}
\usepackage{array}
\usepackage{url}

\usepackage{hyperref}
\hypersetup{
    colorlinks,
    linkcolor={red!50!black},
    citecolor={blue!50!black},
    urlcolor={blue!80!black}
}

\usepackage[nameinlink]{cleveref}
\usepackage{bm}

\theoremstyle{plain}
\newtheorem{lemma}{Lemma}[section]
\newtheorem{proposition}[lemma]{\textbf{Proposition}}
\newtheorem{theorem}[lemma]{\textbf{Theorem}}

\newtheorem{corollary}[lemma]{\textbf{Corollary}}

\theoremstyle{definition}
\newtheorem{definition}[lemma]{\textbf{Definition}}

\newcommand{\R}{\mathbb{R}}
\newcommand{\C}{\mathbb{C}}
\newcommand{\p}{\mathbb{P}}

\newcommand{\South}{{\bm{s}}}
\newcommand{\North}{{\bm{n}}}
\newcommand{\vertOne}{{\bm{w}_1}}
\newcommand{\vertTwo}{{\bm{w}_2}}

\newcommand{\Moebius}{M}
\newcommand{\Fin}{\operatorname{Fin}}

\newcommand{\projcolon}{\!:\!}

\newcommand{\Econst}{E_\mathrm{const}}
\newcommand{\Ered}{\Econst^\mathrm{red}}

\newcommand{\aug}{\mathrm{aug}}
\newcommand{\Gaug}{G_\aug}
\newcommand{\Eaug}{E_\aug}
\newcommand{\Waug}{W^\aug}
\newcommand{\Yaug}{Y^\aug}
\newcommand{\Zaug}{Z^\aug}

\colorlet{ecol}{black!50!white}
\definecolor{colR}{rgb}{.932,.172,.172} 
\definecolor{colB}{rgb}{.255,.41,.884} 
\colorlet{colG}{Goldenrod2}

\colorlet{col1}{DarkGoldenrod}
\colorlet{col2}{DarkKhaki}
\colorlet{col3}{MediumPurple}
\colorlet{col4}{IndianRed}
\colorlet{col5}{SkyBlue!90!black}
\colorlet{col6}{MediumSeaGreen}

\tikzstyle{vertex}=[circle, draw, fill=black, inner sep=0pt, minimum size=4pt]
\tikzstyle{svertex}=[circle, draw, fill=black, inner sep=0pt, minimum size=2pt] 
\tikzstyle{edge}=[line width=1.5pt,ecol]
\tikzstyle{redge}=[edge,colR]
\tikzstyle{bedge}=[edge,colB]

\tikzstyle{labelsty}=[font=\scriptsize]

\tikzstyle{dedge}=[
	edge,%
	postaction={%
		decoration={markings,
			mark=at position 0.5
			with {
				\node[dart,fill,inner sep=1pt,minimum size=6pt,transform shape] (0,0) {};	
			},
		},%
	decorate}%
]

\tikzstyle{sdedge}=[
	edge,%
	postaction={%
		decoration={markings,
			mark=at position 0.5
			with {
				\node[dart,fill,inner sep=1pt,minimum size=4pt,transform shape] (0,0) {};
			},
		},%
	decorate}%
]

\title{Zero-sum cycles in flexible non-triangular polyhedra}
\date{}

\author{%
Matteo Gallet
\and
Georg Grasegger
\and
Jan Legersk\'y
\and
Josef Schicho
}

\begin{document}

\maketitle

\begin{abstract}
 Finding necessary conditions for the geometry of flexible polyhedra is a classical problem that has received attention also in recent times.
 For flexible polyhedra with triangular faces, we showed in a previous work the existence of cycles with a sign assignment for their edges, 
 such that the signed sum of the edge lengths along the cycle is zero. 
 In this work, we extend this result to flexible non-triangular polyhedra.
\end{abstract}

\section*{Introduction}

This work is a generalization of the results in \cite{Gallet2021}.
There, we considered the case of flexible triangular polyhedra, namely polyhedra whose faces are triangles.
We showed that, for each edge whose dihedral angle changes during the flex, there exists a cycle
passing through this edge and a sign assignment for the cycle edges such that the signed sum
of their lengths is zero. More precisely, we proved the following result.

\begin{theorem}
	Consider a polyhedron with triangular faces that admits a flex,
	i.\,e., a continuous deformation preserving the shapes of all faces.
	Let $\{ \vertOne, \vertTwo \}$ be an edge, and let $\South$ and $\North$ be the two vertices adjacent to both~$\vertOne$ and~$\vertTwo$. 
	If the dihedral angle between the faces $\{\vertOne,\vertTwo, \South \}$ and $\{\vertOne,\vertTwo, \North\}$ is not constant along the flex,
	then there is an induced cycle of edges containing $\{ \vertOne, \vertTwo \}$ but neither the vertex~$\South$ nor~$\North$ 
	and there is a sign assignment such that the signed sum of
	lengths of the edges in the cycle is zero.
\end{theorem}

Here, we generalize this theorem to non-triangular flexible polyhedra.
To do so, we first prove that in the triangular case 
the cycle can be chosen so that it only passes through edges 
whose dihedral angle changes during the flex (\Cref{proposition:avoid}). 
Once we have this, it is enough to triangulate the faces of the polyhedron
to obtain the final result.

We refer to \cite{Gallet2021} for a description of the background of the problem,
related results, and the corresponding bibliography.
Moreover, we rely on \cite{Gallet2021} for several results that are used in this paper.
\Cref{preliminaries} recalls all important aspects of \cite{Gallet2021} and can be skipped when the reader is familiar with that paper.
The main result of this paper (\Cref{theorem:nontriangular}) is proven in \Cref{cycles}.

\section{Preliminaries}
\label{preliminaries}

We recall the main definitions and results from \cite{Gallet2021} about flexes of polyhedra.

\begin{definition}
	Let $G = (V, E)$ be the $1$-skeleton of a triangular polyhedron. 
	\begin{itemize}
	\item
	A \emph{realization} of the polyhedron is a map $\rho \colon V \longrightarrow \R^3$ with $\rho(u) \neq \rho(v)$
	for every $\{u, v\} \in E$.
	\item
	We define \emph{edge lengths} induced by a given a realization~$\rho$ by $\lambda_{\{i,j\}} := \left\| \rho(i) - \rho(j) \right\| \in \R_{>0}$ for $\{i,j\} \in E$.
	\item
	We say that two realizations~$\rho_1$ and~$\rho_2$ are \emph{congruent} 
	if there exists an isometry~$\sigma$ of~$\R^3$ such that $\rho_1 = \sigma \circ \rho_2$.
	\item
	A \emph{flex} of a realization~$\rho$ is a continuous map
	$f \colon [0, 1) \longrightarrow (\R^3)^V$ such that 
	\begin{itemize}
	\item
	 $f(0)$ is the given realization~$\rho$;
	\item
	 for any $t \in [0,1)$, the realizations $f(t)$ and $f(0)$ induce the same edge lengths;
	\item
	 for any two distinct $t_1, t_2 \in [0,1)$, the realizations~$f(t_1)$ and~$f(t_2)$ are not congruent.
	\end{itemize}
	\end{itemize}
\end{definition}

To understand the proofs in the next section,
we need to recall the setting from~\cite{Gallet2021}.
We briefly introduce the notions and the results,
so we refer to \cite{Gallet2021} for a more precise account.
From now on, we fix $G=(V,E)$ to be the 1-skeleton of a triangular polyhedron.
We suppose that there is a realization~$\rho_0$ that admits a flex, 
and that the dihedral angle at the edge $\{\vertOne, \vertTwo\}$ changes during the flex.
The two neighbors of $\vertOne$ and $\vertTwo$ are the vertices $\South$ and $\North$.

We encode realizations of the polyhedron as follows. 
We consider the three-dimensional variety $\Moebius \subset \p^4$ defined by the equation
\begin{equation*}
  x^2+y^2+z^2-rh = 0 \,.
\end{equation*}
The variety $\Moebius$ contains a copy of $\R^3$, namely the image of $(x,y,z) \mapsto (x \projcolon y \projcolon z \projcolon x^2+y^2+z^2 \projcolon 1)$.
Hence, a realization of the polyhedron is given by an element in $\Moebius^V$.
Since $\Moebius$ is an algebraic variety, we can consider its extension to the complex numbers.
From now on, we also take into account complex realizations of the polyhedron.

Given a vector of edge lengths $\lambda$, we consider the complex extension of the variety of real realizations inducing $\lambda$.
Namely, we define $W$ to be the set of maps $\rho \colon V \longrightarrow \C^3$ such that, if we write $\rho(v_i) = (x_i, y_i, z_i)$,
\begin{equation*}
  (x_1 - x_2)^2 + (y_1 - y_2)^2 + (z_1 - z_2)^2 = \lambda^2_{\{v_1, v_2\}}
\end{equation*}
for all $\{v_1, v_2\} \in E$.
Whenever a realization belongs to $W$ then all its congruent ones also belong to $W$.
To select representatives for each congruence class in $W$, 
we pin the triangle $\vertOne,\vertTwo,\North$ to the initial realization~$\rho_0$
and so we define
\begin{equation*}
  Z := 
 \bigl\{
  \rho \in W \, | \,
  \rho(\vertOne) = \rho_0(\vertOne), \;
  \rho(\vertTwo) = \rho_0(\vertTwo), \;
  \rho(\North) = \rho_0(\North)
 \bigr\}
 \,. 
\end{equation*}
Then no two elements in this new set $Z$ differ by a direct isometry.
We define~$Y$ to be the image of $Z$ under the embedding $(\C^3)^V \longrightarrow \Moebius^V$.
The existence of a flex implies that $Y$ has positive dimension.
Since we are assuming that the dihedral angle at the edge $\{\vertOne,\vertTwo\}$ changes during the flex,
the projection $Y_\South$ of~$Y$ onto the $\South$-coordinate of $\Moebius^V$ is positive-dimensional.
Therefore, $Y_\South$ intersects the hyperplane $\{h=0\}\subset \p^4$ non-trivially.
The main argument in \cite{Gallet2021} relies on fixing an element $\rho_\infty \in Y$ 
such that $\rho_\infty(\South) \in Y_\South \cap \{h=0\}$.
Here, having extended the setting to the complex numbers proves to be crucial.

We define $\Moebius_\infty$ to be $\Moebius\cap \{h=0\}$.
Then $\rho_\infty(\South)$ belongs to $\Moebius_\infty$.
Using $\rho_\infty$ one defines a coloring of the vertices of $G$ (see \cite[Definition 2.6]{Gallet2021}),
where a vertex $v \in V$ is colored:
\begin{itemize}
 \item \emph{red} if $\rho_{\infty}(v) \in \Moebius \setminus \Moebius_{\infty}$;
 \item \emph{blue} if $\rho_{\infty}(v) = (x_{\South}\projcolon y_{\South}\projcolon z_{\South}\projcolon r_v\projcolon 0)$ 
 where $\rho_{\infty}(\South) = (x_{\South}\projcolon y_{\South}\projcolon z_{\South}\projcolon r_{\South}\projcolon 0)$;
 \item \emph{gold} otherwise.
\end{itemize}
By construction, the vertices $\vertOne$, $\vertTwo$, and $\North$ are red, while $\South$ is blue.
From the coloring, we define a \emph{blue walk} and a \emph{red walk} as follows.
We define an equivalence relation on edges with one blue vertex and one red vertex:
two edges are in relation if they belong to the same triangle (and then we take the reflexive-transitive closure). 
Now, the red vertices, respectively the blue vertices, in the equivalence class of $\{\vertOne, \South\}$ 
create a red walk, respectively blue walk. These walks respectively contain $\{\vertOne, \vertTwo\}$ and $\South$.
Within the red walks, it is possible to find a cycle containing $\{\vertOne, \vertTwo\}$ (see \cite[Lemma 2.10]{Gallet2021}).

The importance of these two walks derives from the fact that we can prove (\cite[Lemma 2.9]{Gallet2021})
that for the vertices $v$ in the red walk we have $\rho_{\infty}(v) \in \Fin_{\rho_{\infty}(\South)}$, 
where for $p = (x_p \projcolon y_p \projcolon z_p \projcolon r_p \projcolon h_p) \in \p^4$ we define
\[
 \Fin_p := \{ (x\projcolon y \projcolon z \projcolon r \projcolon h) \in \Moebius \,|\, 
              h \neq 0 \text{ and }  x \, x_p + y \, y_p + z \, z_p - \frac{1}{2}(r \, h_p + h \, r_p) = 0 \} \,.
\]
The set $\Fin_p$ can be also thought as the intersection of $\Moebius \setminus \Moebius_{\infty}$ with the embedded tangent space of $\Moebius$ at $p$.
The remarkable property of $\Fin_p$, when $p \in \Moebius_{\infty}$, is that for its points, 
the analogue of the Euclidean distance behaves like a distance on a one-dimensional space (see \cite[Lemma 1.11]{Gallet2021});
from the proof of \cite[Theorem~2.2]{Gallet2021}, we get the following result, 
which ensures that the zero-sum property holds for cycles of red vertices.

\begin{lemma}
	\label{lemma:cycle_fin}
	Let $\mathcal{D}=(v_1, \dotsc, v_k, v_{k+1} = v_1)$ be a cycle such that
	$\rho_{\infty}(v_j) \in \Fin_{\rho_{\infty}(\South)}$ for all $j\in\{1,\ldots, k\}$.
	There exist $\eta_j\in\{-1,1\}$ for $j\in\{1,\ldots, k\}$ such that 
	 \[
	  \sum_{j=1}^{k} \eta_j \, \lambda_{\{v_j, v_{j+1}\}}=0\,,
	 \]
	 where $\lambda_{\{v_j, v_{j+1}\}}$ are the edge lengths induced by the realizations in the flex.
\end{lemma}

\section{Cycles in polyhedra}
\label{cycles}

We start strengthening the main result in \cite{Gallet2021}, 
by showing that if the dihedral angle at one edge changes during a flex,
then it is possible to find a cycle that contains this edge but avoids
all edges whose dihedral angles stay constant (\Cref{proposition:avoid}).
To obtain this, we have to sacrifice the property of the cycle to be induced.
In this way, we can extend our result to polyhedra
that are not necessarily triangular (\Cref{theorem:nontriangular}).

\begin{proposition}
	\label{proposition:avoid}
	Let $G$ be the 1-skeleton of a triangular polyhedron with a realization that admits a flex.
	Suppose that all triangles in the realization are non-degenerate, i.e., the vertices of each triangle are non-collinear.
	Then, for every edge of~$G$ whose dihedral angle changes along the flex,
	there is a cycle in~$G$ containing that edge and
	a sign assignment such that the signed sum of the lengths of the edges in the cycle is zero.
	The cycle can be chosen so that it consists only of edges whose dihedral angle changes along the flex.
\end{proposition}

	To prove \Cref{proposition:avoid}, we need to have better control on the constructions we made in \cite{Gallet2021}.
	In doing this, we employ the following notation, which is compatible with the one in \cite[Section 2]{Gallet2021} and in \Cref{preliminaries}:
	\begin{itemize}
	\item let $\rho_0$ be the realization of~$G=(V,E)$ that admits a flex~$f$;
	\item let $\Econst$ be the set of all edges in $E$ whose dihedral angle does not change along the flex~$f$;
	\item let $\{\vertOne,\vertTwo\}$ be an edge for which the dihedral angle is not constant, 
	and let $\South$ and $\North$ be the two opposite vertices of the two triangles containing~$\{\vertOne,\vertTwo\}$.
	\end{itemize}
	The idea is that if an edge belongs to~$\Econst$, then we can locally modify the triangulation of the polyhedron, 
	without changing the set of realizations, so that the edge is avoided by the cycle we construct.
	To do this, we introduce the operation of a \emph{flip}.
	
	\begin{definition}
		\label{definition:flip}
		Let $H=(V_H,E_H)$ be the 1-skeleton of a triangular polyhedron.
		Let $\{v_1,v_2\} \in E_H$ and let $u_1,u_2\in V_H$ be the two opposite
		vertices of the two triangles in~$H$ containing the edge~$\{v_1,v_2\}$ as in \Cref{figure:flip}.
		The \emph{flip of $H$ on the edge}~$\{v_1,v_2\}$ is the $1$-skeleton obtained from $H$
		by replacing the edge~$\{v_1,v_2\}$ by~$\{u_1,u_2\}$, 
		and leaving everything else unchanged.
		Once a sequence of edges is fixed,
		we speak of a flip on the sequence by iteratively applying flips.
		If the sequence is empty, then the flip is the graph itself.
	\end{definition}
	
\begin{figure}[H]
    \centering
	\begin{tikzpicture}[]
        \begin{scope}
            \node[vertex,label={[labelsty]above:$v_1$}] (v1) at (0,1) {};
            \node[vertex,label={[labelsty]below:$v_2$}] (v2) at (0,-1) {};
            \node[vertex,label={[labelsty]left:$u_1$}] (u1) at (-1,0) {};
            \node[vertex,label={[labelsty]right:$u_2$}] (u2) at (1,0) {};
            
            \draw[edge] (v1)edge(v2) (v2)edge(u1) (u1)edge(v1);
            \draw[edge] (v1)edge(v2) (v2)edge(u2) (u2)edge(v1);
        \end{scope}
        \begin{scope}[xshift=3cm]
            \node[] (o) at (0,0) {$\stackrel{\mathrm{flip}}{\resizebox{0.8cm}{!}{$\rightsquigarrow$}}$};
        \end{scope}
        \begin{scope}[xshift=6cm]
            \node[vertex,label={[labelsty]above:$v_1$}] (v1) at (0,1) {};
            \node[vertex,label={[labelsty]below:$v_2$}] (v2) at (0,-1) {};
            \node[vertex,label={[labelsty]left:$u_1$}] (u1) at (-1,0) {};
            \node[vertex,label={[labelsty]right:$u_2$}] (u2) at (1,0) {};
            
            \draw[edge] (u1)edge(u2) (u2)edge(v1) (v1)edge(u1);
            \draw[edge] (u1)edge(u2) (u2)edge(v2) (v2)edge(u1);
        \end{scope}
	\end{tikzpicture}
	\caption{The flip on the edge $\{v_1,v_2\}$.}
	\label{figure:flip}
\end{figure}
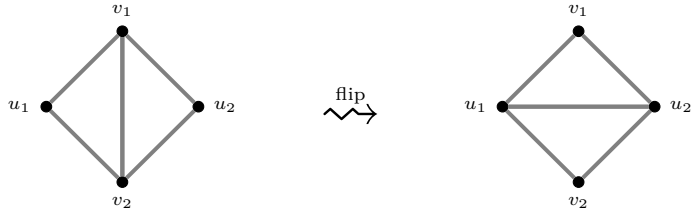

	To illustrate the idea of the proof, suppose that $\{v_1, v_2\} \in \Econst$ and
	let $G'$ be the flip of~$G$ on the edge~$\{v_1,v_2\}$.
	The realization~$\rho_0$ is a realization of~$G'$ as well.
	Moreover, the assumption that the dihedral angle at $\{v_1,v_2\}$ is constant along the flex
	guarantees that the flex~$f$ of~$(G,\rho_0)$ is also a flex of~$(G',\rho_0)$.
	In particular,
	we can obtain the \emph{same} element~$\rho_\infty$, as in \Cref{preliminaries}, for both~$G$ and~$G'$ 
	and hence the same coloring of vertices.
	Clearly, the edge sets of~$G$ and~$G'$ are different and so are the triangles with red and blue vertices.
	Hence, we may get different red and blue walks.
	If the edge $\{v_1,v_2\}$ is in the red walk in~$G$, and hence
	possibly in the zero-sum cycle, we can flip to $G'$, where $v_1$ and $v_2$ are non-adjacent,
	and obtain a red walk avoiding $\{v_1,v_2\}$.  
	
	We start to give more details following the blueprint we have just presented.
	The first aim is to construct an element~$\rho_\infty$ valid for all flips on edges in $\Econst$ (\Cref{lemma:same_rho}).
	In order to do that, we augment the graph $G=(V,E)$ with edges between pairs of vertices
	whose distance is fixed in the flex.
	
	\begin{definition}
	Let $\Gaug$ be the graph with the vertex set~$V$ and edge set 
	\begin{equation*}
		\Eaug := 
		E \cup 
		\big\{
			\{u,v\} \,\mid\, u, v\in V, \, u\neq v \text{ and } \|\rho(u)-\rho(v)\| \text{ is constant along } f
		\big\}\,.
	\end{equation*}
	\end{definition}
	
	\begin{lemma}
		\label{lemma:properties_flip}
		Let $G'$ be the flip of~$G$ on any sequence of edges in~$\Econst$.
		Then 
		\begin{enumerate}
		\renewcommand{\theenumi}{(\alph{enumi})}
		\renewcommand{\labelenumi}{\theenumi}
		\item $\rho_0$ is a realization of~$G'$;\label{enum:realization}
		\item $f$ is a flex of $(G', \rho_0)$;\label{enum:flex}
		\item the dihedral angle at any edge in~$\Econst$ that is an edge in~$G'$ is constant along~$f$;\label{enum:constant_angle}
		\item $G'$ is a subgraph of $\Gaug$.\label{enum:subgraph}
		\end{enumerate}
	\end{lemma}
	\begin{proof}
		The fact that $\rho_0$ is a realization of~$G'$ follows immediately from the fact that the vertices of~$G$ and~$G'$ are the same.
		We now show the last three properties by induction on the number of flips.  
		The base case of the induction (no flips) is trivial. 
		Suppose we performed flips on $e_1, \dotsc, e_k \in \Econst$ (in this order), obtaining $G'_{k}$, 
		and now we perform a further flip on $e_{k+1} = \{v_1, v_2\} \in \Econst$, obtaining~$G'_{k+1}$. 
		In particular, this means that $e_{k+1}$ is an edge in~$G'_{k}$. 
		Let $u_1, u_2$ be the two opposite vertices in the two triangles in~$G'_{k}$ containing~$e_{k+1}$. 
		Notice that to show that $f$ is a flex of $(G'_{k+1}, \rho_0)$, 
		we only have to prove that all realizations of~$f$ determine the same lengths of the edge~$\{u_1, u_2\}$. 
		By inductive assumption, we know that $f$ is a flex of~$(G'_{k}, \rho_0)$ and that the dihedral angle at~$e_{k+1}$ is constant along~$f$. 
		Then, we know that the distance between~$u_1$ and~$u_2$ is constant along~$f$.
		Therefore, $f$ is a flex of~$G'_{k+1}$ and $G'_{k+1}$ is a subgraph of~$\Gaug$.
		We are left to show that all edges in~$\Econst$ that are also edges of~$G'_{k+1}$ have constant dihedral angle along~$f$. 
		The induction hypothesis covers all those edges that are not in the $4$-cycle $(v_1, u_1, v_2, u_2, v_1)$. 
		Without loss of generality, it is enough to prove that the property holds for~$\{v_1, u_2\}$.
		Let $v$ be the vertex such that $v$ and $v_2$ are the opposite vertices
		of the two triangles in $G'_{k}$ containing $\{v_1, u_2\}$.
		Since the dihedral angles at $\{v_1, v_2\}$ and $\{v_1, u_2\}$ are constant
		along the flex~$f$ of~$(G'_{k},\rho_0)$,
		the distance between $u_1$ and $v$ is constant along the flex $f$.
		Hence, the dihedral angle between the faces $\{v_1,u_2,u_1\}$ and $\{v_1,u_2,v\}$ in~$G'_{k+1}$ is constant along the flex~$f$ of~$(G'_{k+1},\rho_0)$.
	\end{proof}
	
	\begin{lemma}
		\label{lemma:same_rho}
		There exists an element $\rho_{\infty} \in \Moebius^V$ such that  
		for any flip~$G'$ of~$G$ on any sequence of edges in~$\Econst$,
		we have that $\rho_{\infty} \in Y'$ and $\rho_{\infty}(\South) \in \Moebius_{\infty}$,
		where $Y'$ is the subvariety of~$\Moebius^V$ obtained from~$G'$ and~$\rho_0$ as in \Cref{preliminaries}.
	\end{lemma}
	\begin{proof}
		Let $\Waug$ be the set of realizations $\rho \colon V \longrightarrow \C^3$ such that 
		\[
			(x_u - x_v)^2 + (y_u - y_v)^2 + (z_u - z_v)^2 = 
			\left\| \rho_0(u) - \rho_0(v) \right\|^2
			\text{ for all } \{u, v\} \text{ in } \Eaug \,,
		\]
		where $(x_{u}, y_{u}, z_{u})$ are the coordinates of the point~$\rho(u)$, and similarly for~$\rho(v)$.
		Pick any flip~$G'$ of~$G$ on any sequence of edges in~$\Econst$.
		Let $W'$ be the algebraic set constructed starting from~$G'$ as in \Cref{preliminaries}.
		From \Cref{lemma:properties_flip}\ref{enum:subgraph}, we know that $\Waug \subset W'$.
		
		Starting from $\Waug$, we can construct $\Zaug \subset Z'$ and $\Yaug \subset Y'$
		as in \Cref{preliminaries} since $\rho_0$ is a realization of~$\Gaug$.
		Because of the definition of~$\Waug$ and the existence of the flex~$f$,
		the projection~$\Yaug_\South$ of~$\Yaug$ on the copy of~$\Moebius$ indexed by the vertex~$\South$
		is still positive-dimensional; in fact, the extra constraints of~$\Waug$ are introduced only for pairs of vertices
		whose distance is constant during the flex.

		Therefore, we can pick an element~$\rho_{\infty}$ in~$\Yaug$ so that $\rho_{\infty}(\South) \in \Moebius_{\infty}$.
		This element~$\rho_{\infty}$ satisfies then the requirements of the statement.
	\end{proof}

	We use $\rho_{\infty}$ from \Cref{lemma:same_rho} to color the vertices of~$G$ according to \Cref{preliminaries}.
	For any flip $G'$ of~$G$ on a sequence of edges in~$\Econst$,
	we get the blue and red walk in $G'$ via~$\rho_\infty$.
	
	\begin{definition}
		A red vertex~$v$ of~$G$ is called \emph{red-achievable} 
		if there exists a flip~$G'$ of~$G$ on a sequence of edges in~$\Econst$
		such that $v$ is in the red walk in~$G'$ via~$\rho_\infty$.
	\end{definition}
	By \cite[Lemma 2.9]{Gallet2021}, we get the following statement.
	\begin{corollary}
		\label{cor:red_achievable}
		If $v$ is a red-achievable vertex of $G$,
		then $\rho_\infty(v)\in \Fin_{\rho_{\infty}(\South)}$.
	\end{corollary}
	
	The last tool we need to prove \Cref{proposition:avoid} is a result that ensures
	that performing flips allows us to get rid of edges with constant dihedral angles
	in the zero-sum cycle we are looking for.
	
	\begin{lemma}
		\label{lemma:red_const_edge}
		Let $\{v_1,v_2\}$ be an edge in $\Econst$.
		Let $u_1,u_2$ be the two opposite vertices of 
		the two triangles in $G$ containing
		the edge $\{v_1,v_2\}$.
		If $v_1$ and $v_2$ are red-achievable,
		then neither $u_1$ nor $u_2$ is red-achievable.
	\end{lemma}
	\begin{proof}
		Suppose that $u_i$ is red-achievable.
		By \Cref{cor:red_achievable} and \Cref{lemma:cycle_fin}, the triangle $\{v_1,v_2,u_i\}$ in $G$
		would be degenerate in the realizations of the flex~$f$, which contradicts the assumption of \Cref{proposition:avoid}. 
	\end{proof}
	
\begin{proof}[\Cref{proposition:avoid}]
	Our goal is to find a cycle of $G$ containing $\{\vertOne,\vertTwo\}$
	such that none of its edges is in $\Econst$
	and for each of its vertices~$v$ we have $\rho_\infty(v)\in \Fin_{\rho_{\infty}(\South)}$.
	Then the statement follows from \Cref{lemma:cycle_fin}.
	We do so using flips.
	
	Let $\Ered$ be the set of all edges $\{v_1,v_2\}\in \Econst$
	such that $v_1$ and $v_2$ are red-achievable.
	By \Cref{lemma:red_const_edge}, every edge $\{v_1,v_2\}\in \Ered$
	is a diagonal of a 4-cycle in~$G$ consisting of edges avoiding $\Ered$.
	This implies that there is a unique graph~$G'$ that is the flip
	of $G$ on any permutation of the edges in $\Ered$.
	None of the edges in~$\Ered$ is an edge of $G'$ and
	no edge introduced by the flips has vertices that are both red-achievable by \Cref{lemma:red_const_edge}.
	Let $R'$ be the red walk in $G'$ via $\rho_\infty$ constructed as in \Cref{preliminaries}.
	Since all vertices of $R'$ are red-achievable by definition,
	no edge in~$\Econst$ is in $R'$
	and all edges of $R'$ are edges of~$G$.
	Following the argument in \cite[Lemma 2.10]{Gallet2021}, there is a cycle in $G'$ containing $\{ \vertOne, \vertTwo \}$ 
	such that all its edges are in $R'$, i.e., it is a cycle in $G$.
	By \Cref{lemma:cycle_fin}, it is a zero-sum cycle.
\end{proof}

This stronger result allows us to extend the result of \cite{Gallet2021}
to polyhedra whose faces are not necessarily triangles.
The notion of realization extends immediately to this new setting.
The notion of flex, instead, requires a little bit of care:
what we are interested in here are, in fact, flexes for which 
the faces \emph{do not change their shapes}.
Notice that, by a well-know theorem of Cauchy, 
with this notion of flex convex realizations of polyhedra do not admit flexes.

\begin{definition}
	Let $H = (V,E,F)$ be the $2$-skeleton of a polyhedron,
	namely, $(V,E)$ is a graph and $F$ is a set faces, i.e., cycles in $(V,E)$,
	such that every edge in $E$ is in exactly two faces. 
	A \emph{realization} of~$H$ is a map $\rho \colon V \longrightarrow \R^3$ 
	such that $\rho(u) \neq \rho(v)$ for every $\{u, v\} \in E$.
	A \emph{flex} of a realization~$\rho$ of~$H$ is a continuous map $f \colon [0,1) \longrightarrow (\R^3)^V$ such that
	\begin{itemize}
	 \item $f(0)$ is the given realization~$\rho$;
	 \item for any $t \in [0,1)$ and for every face~$g \in F$, the images of~$g$ under the realizations~$f(0)$ and~$f(t)$ are congruent;
	 \item for any two distinct $t_1, t_2 \in [0,1)$, the realizations $f(t_1)$ and $f(t_2)$ are not congruent.
	\end{itemize}
\end{definition}

\begin{theorem}
	\label{theorem:nontriangular}
	Let $H$ be the $2$-skeleton of a polyhedron with a realization that admits a flex.
	Suppose that there is a triangulation of the faces of the polyhedron such that
	the vertex set of the triangulation is the same as the one of $H$ and
	all triangles in the realization are non-degenerate, i.e., the vertices of each triangle are non-collinear.
	Then, for every edge of~$H$ whose two incident faces change their relative position along the flex,
	there is a cycle in~$H$ containing that edge and
	a sign assignment such that the signed sum of the lengths of the edges in the cycle is zero.
	Moreover, the two incident faces of every edge in the cycle change their relative position along the flex.
\end{theorem}
\begin{proof}
	Let $G$ be the $1$-skeleton of the triangulation of the polyhedron in the statement.
	Let $f$ be the flex of~$H$. By construction, $f$ is also a flex of~$G$.
	Now, $G$ and $f$ satisfy the assumptions of \Cref{proposition:avoid}.
	Moreover, by construction all edges in~$G$ that are diagonals of faces of~$H$ have constant dihedral angle along the flex.
	Thus, the cycle provided by \Cref{proposition:avoid} consists only of edges in~$H$
	and it can be chosen so that the dihedral angles of all its edges change along the flex.
	Hence, this cycle fulfills the requirements of the statement.
\end{proof}

\paragraph{Acknowledgments.}

Josef Schicho and Jan Legersk\'y have been supported by the Austrian Science Fund (FWF): P31061.
Jan Legersk\'y has been supported by the Ministry of Education, Youth and Sports of the Czech Republic, project no. CZ.02.1.01/0.0/ 0.0/16\_019/0000778.
Matteo Gallet has been supported by the Austrian Science Fund (FWF): Erwin Schr\"odinger Fellowship J4253.
Georg Grasegger has been supported by the Austrian Science Fund (FWF): P31888.

\bigskip
\bigskip

\textsc{(MG, GG) Johann Radon Institute for Computation and Applied Mathematics (RICAM), Austrian
Academy of Sciences, Austria}\\
Email address: \texttt{matteo.gallet@ricam.oeaw.ac.at}, \\
\phantom{Email address: }\texttt{georg.grasegger@ricam.oeaw.ac.at}

\textsc{(JL) Department of Applied Mathematics, Faculty of Information Technology, Czech Technical University in Prague, Czech Republic}\\
Email address: \texttt{jan.legersky@fit.cvut.cz}

\textsc{(JS) Johannes Kepler University Linz, Research Institute for Symbolic Computation (RISC), Austria}\\
Email address: \texttt{jschicho@risc.jku.at}

\end{document}